\documentclass[11pt,leqno]{amsart}
\usepackage{amsmath}
\usepackage{amsfonts}
\usepackage{amssymb}
\usepackage{graphicx}
\usepackage{color}
\usepackage{hyperref}
\usepackage{xcolor}

\newtheorem{theorem}{Theorem}[section]
\newtheorem*{theorem*}{Theorem}
\newtheorem{lemma}{Lemma}[section]
\newtheorem{corollary}[theorem]{Corollary}
\newtheorem{proposition}{Proposition}[section]

\newtheorem{definition}[theorem]{Definition}

\newtheorem{remark}[theorem]{Remark}

\def\l{\lambda}

\def\p{\partial}

\def\R{\mathbb{R}}

\def\vp{\varphi}

\def\k{\kappa}

\numberwithin{equation}{section}

\begin{document}

\title[Eigenvalues estimates on  quaternionic K\"ahler manifolds]{Lower bounds for the first eigenvalue of the $p$-Laplacian on quaternionic K\"ahler manifolds}

\author{Kui Wang} \thanks{The research of the first author is supported by NSF of Jiangsu Province No. SBK2023020968}
\address{School of Mathematical Sciences, Soochow University, Suzhou, 215006, China}
\email{kuiwang@suda.edu.cn}

\author{Shaoheng Zhang} \thanks{}
\address{School of Mathematical Sciences, Soochow University, Suzhou, 215006, China}
\email{20234007008@stu.suda.edu.cn}



\begin{abstract}
We study the first nonzero eigenvalues for the $p$-Laplacian on quaternionic K\"ahler  manifolds. 
Our first result is a lower bound for the first nonzero closed (Neumann) eigenvalue of the $p$-Laplacian on compact  quaternionic K\"ahler manifolds. Our second result is a lower bound for the first Dirichlet eigenvalue of the $p$-Laplacian on compact quaternionic K\"ahler manifolds with smooth boundary. Our results generalize corresponding results for the Laplacian eigenvalues on quaternionic K\"ahler manifolds proved in \cite{LW23}.

\mbox{}

\end{abstract}

\maketitle

\section{Introduction and Main Results}

    Let $(M^m,g)$ be an $m$-dimensional compact  manifold,  and the $p$-Laplacian of the metric $g$ is defined by
    \begin{align*}
    \Delta_p u:=\operatorname{div}(|\nabla u|^{p-2}\nabla u)
    \end{align*}
    for $p\in (1,\infty)$. The $p$-Laplace eigenvalue equation is
    \begin{align*}
    -\Delta_p u(x)=\lambda |u(x)|^{p-2}u(x), \quad x\in M,
    \end{align*}
    and $\l\in \R$ is called a closed eigenvalue  when $\p M=\emptyset$, a Dirichlet eigenvalue if $u(x)=0$ on $\p M$ when $\p M\neq \emptyset$, a Neumann eigenvalue  if $\p_\nu u(x)=0$ on $\p M$ when $\p M\neq \emptyset$. Here $\nu$ denotes the unit outer  normal to $\p M$.

Estimates of the first nonzero eigenvalues have been extensively studied in both mathematics and physics for a long time.  In the past few decades, a number of important results on lower bounds in terms of geometric data were obtained by various authors (\cite{AC11, BQ00,BS19,Cha84,DSW21,HWZ20,LL10,RSW20,YS94,  SWW19}). 
     One of the classical results is the following optimal lower bound for the first nonzero eigenvalue of the Laplacian.
    \begin{theorem}
    \label{thm classical}
        Let $(M^m, g)$ be a compact Riemannian manifold (possibly with a smooth convex boundary) with diameter $D$ and $\operatorname{Ric} \geq (m-1)\k$ for $\k\in \R$. 
        Let $\mu_{1}$ be the first nonzero eigenvalue of the Laplacian on $M$ (with Neumann boundary condition if $\p M \neq \emptyset$). Then
    \begin{align*}
        \mu_{1}\geq \bar{\mu}_1(m,\k,D),
        \end{align*}
        where $\bar{\mu}_1(m,\k,D)$ is the first nonzero Neumann eigenvalue of the one-dimensional eigenvalue problem
        \begin{align*}
        \vp'' - (m-1)T_{\k}\vp' =-\mu\vp
        \end{align*}
        on the interval $[-D/2, D/2]$, and $T_{\k}$ is defined in (\ref{eqn Tk}).
    \end{theorem}
    
    For $\k>0$,  Lichnerowicz's estimate \cite{Lich58} asserts that if $M$ is closed  with $\text{Ric} \geq (m-1)\k > 0$, then $\mu_1 \ge m\k$. 
    Escobar \cite{Escobar90} proved that the same lower bound as in Lichnerowicz's result holds if $M$ is compact with a  convex boundary. 
    If $\k=0$, Zhong and Yang \cite{ZY84} proved the sharp lower bound that $\mu_1 \geq \pi^2/D^2$, by refining the gradient estimates of Li \cite{Li79} and Li and Yau \cite{LY80}. 
    For general $\k \in \R$, Theorem \ref{thm classical} was proved by Kr\"oger \cite{Kr92} using the gradient estimates method, and  by Chen and Wang \cite{CW94,CW95} using probabilistic coupling method independently.
    Recently, Andrews and Clutterbuck \cite{AC13} gave a simple proof via the modulus of continuity estimates for the solutions to the heat equation, and Zhang and the first author of this paper \cite{ZW17} gave an elliptic proof based on \cite{AC13} and \cite{Ni13}.

    On compact quaternionic K\"ahler manifolds, Li and the first author of this paper\cite{LW23} proved the following lower bounds of the first nonzero eigenvalue of the Laplacian.
    \begin{theorem}[\cite{LW23}]
    \label{thm 1.1}
    Let $(M^m, g, I, J, K)$ be a compact quaternionic K\"ahler manifold (possibly with a smooth strictly convex boundary) of quaternionic dimension $m \geq 2$ and diameter D. Suppose that the scalar curvature of $M$ is bounded from below by $16 m(m+2) \k$ for some $\k \in \R$. Let $\mu_1$ be the first nonzero eigenvalue of the Laplacian on $M$ (with Neumann boundary condition if $\partial M \neq \emptyset)$. Then
   \begin{align*}
    \mu_{1} \geq \bar{\mu}_1(m, \k, D),
    \end{align*}
    where $\bar{\mu}_1(m, \k, D)$ is the first nonzero Neumann eigenvalue of the one-dimensional eigenvalue problem
    \begin{align} \label{thm 1.2}
    \vp''-(4(m-1) T_\k+3 T_{4 \k}) \vp'=-\lambda \vp
    \end{align}
    on $[-D/2, D/2]$, where $T_{\k}(t)$ is defined in (\ref{eqn Tk}).
    \end{theorem}
    
On compact K\"ahler manifolds, Li and the first author of this paper \cite{LW21} proved lower bounds of the first nonzero eigenvalue of the Laplacian, which was generalized by the authors \cite{WZ23} to the $p$-Laplacian for $1<p\le 2$. Blacker and Seto \cite{BS19} proved a lower bound for the first nontrivial eigenvalue of the $p$-Laplacian on K\"ahler manifolds assuming positive Ricci curvature. Recently, Rutkowski and Seto \cite{RS23}  established an explicit lower bound for the first eigenvalue of the Laplacian on K\"ahler manifolds based off the comparison results of Li and the first author of this paper \cite{LW21}. It is thus a natural question to study the first nonzero eigenvalue of the $p$-Laplacian on quaternionic K\"ahler manifolds. Regarding this, we prove  
    \begin{theorem}
    \label{thm 1.3}
    Let $(M^{m},g,I,J,K)$ be a compact quaternionic K\"{a}hler manifold of quaternionic dimension $m\ge 2$ and diameter $D$, whose scalar curvature is bounded from below by $16m(m+2)\k$ for some $\k \in \R$.
    Let $\mu_{1,p}(M)$ be the first nonzero eigenvalue of the $p$-Laplacian on M (with Neumann boundary condition if $M$ has a strictly convex boundary). Assume $1<p\leq 2$, then
    \begin{align*}
    \mu_{1,p}(M) \geq \bar{\mu}_{1}(m,p,\k,D),
    \end{align*}
    where $\bar{\mu}_{1}(m,p,\k,D)$ is the first nonzero Neumann eigenvalue of the one-dimensional eigenvalue problem
    \begin{align}
    \label{1.3}
    (p-1)|\vp'|^{p-2}\vp''
    -( 4(m-1)T_{\k} +3T_{4\k} )|\vp'|^{p-2}\vp'
    = -\mu |\vp|^{p-2} \vp
    \end{align}
    on interval $[-D/2,D/2]$, where $T_{\k}(t)$ is defined in (\ref{eqn Tk}).
    \end{theorem}
	
    Note that if $\k=0$, ODE \eqref{1.3} can be solved by using $p$-trigonometric functions, and we get an explicit lower bound.
    \begin{corollary}
    With the same assumptions as in Theorem \ref{thm 1.3}, and assume further that $\k=0$. Then
    \begin{equation}\label{1.4}
		\mu_{1,p}(M) \geq (p-1)(\frac{\pi_p}{D})^{p},
    \end{equation}
    where $\displaystyle \pi_p=\frac{2 \pi}{p \sin(\pi /p)}$.
    \end{corollary}
    
    Now let us turn to lower bounds for the Dirichlet eigenvalues of the Laplacian. For  $\k, \Lambda \in \R$, denote by $C_{\k,\Lambda}$ the unique solution of the initial value problem
    \begin{align}
	\label{1.5}
		\left \{
		\begin{aligned}
			&\phi''(t)+\kappa \phi(t)=0,\\
			&\phi(0)=1,\ \phi'(0)=-\Lambda,
		\end{aligned}
		\right.
	\end{align}
    and denote $T_{\kappa,\Lambda}$ for
	$\kappa, \Lambda \in \mathbb{R}$ by
	\begin{align}
	\label{1.6}
	    T_{\kappa,\Lambda}(t):
	    =-\frac{ C_{\kappa,\Lambda}'(t) }
	{ C_{\kappa,\Lambda}(t)}.	
	\end{align}
       Recall  the inradius of $M$ is  defined by
        \[
        R=\displaystyle \sup_{x \in M} d(x, \partial M),
        \] 
        where $d(x,\p M)$ denotes the distance function to the boundary $\p M$ given by
        \[
        d(x,\partial M)
        =\inf_{y \in \partial M} d(x,y).
        \]
    In the Riemannian setting, the following result is well-known.
    \begin{theorem}
    \label{thm 1.5}
        Let $(M^m, g)$ be a compact Riemannian manifold with smooth boundary $\p M \neq \emptyset$. Suppose that the Ricci curvature of $M$ is bounded from below by $(m-1)\k$ and the mean curvature of $\p M$ is bounded from below by $(m-1)\Lambda$ for some $\k,\Lambda \in \R$. Let $\lambda_1$ be the first Dirichlet eigenvalue of the Laplacian on $M$. Then
        \[
        \lambda_1 \geq \bar{\lambda}_1(m,\k,\Lambda,R),
        \]
        where $\bar{\lambda}_1(m,\k,\Lambda,R)$ is the first eigenvalue of the one-dimensional eigenvalue problem
        \[
	\left\{
	\begin{array}{l}
		\vp''-(m-1) T_{\k, \Lambda} \vp'=-\lambda \vp, \\
		\vp(0)=0,\ \vp'(R)=0.
	\end{array}
	\right.
	\]
        Here $R$ denotes the inradius of $M$.
    \end{theorem}
    Regarding Theorem \ref{thm 1.5}, the special case $\k=\Lambda=0$ is due to Li and Yau \cite{LY80} and the general case is obtained by Kasue \cite{Kas84}.

    In \cite{LW23}, Li and the first author of this paper obtained an analogous theorem in the quaternionic K\"ahler setting.
    \begin{theorem}[\cite{LW23}]
        Let $(M^m, g, I,J,K)$ be a compact quaternionic K\"{a}hler manifold with smooth  boundary $\partial M$ and of quaternionic dimension $m\geq 2$.
		Suppose that the scalar curvature is bounded from below by $16m(m+2)\kappa$ for some $\kappa \in \R$, and the second fundamental form on $\partial M$ is bounded from below by $\Lambda \in \R$.
		Let $\lambda_{1}$ be the first Dirichlet eigenvalue of the Laplacian on $M$. Then
		\begin{align}
		\lambda_{1} \geq \bar{\lambda}_1(m,\k,\Lambda, R),
		\end{align}
		where $\bar{\lambda}_1(m,\k,\Lambda, R)$ is the first eigenvalue of the one-dimensional eigenvalue problem
        \[
        \begin{cases}
        \vp''-(4(m-1)T_{\k,\Lambda}+3T_{4\k,\Lambda})\vp'=-\bar{\lambda}_1\vp,\\
        \vp(0)=0,\vp'(R)=0,
        \end{cases}
        \]
        where $R$ denotes the inradius of $M$.
    \end{theorem}

    Our second main result is the following $p$-Laplace analogue of the lower bounds for the first Dirichlet eigenvalue on quaternionic K\"ahler manifolds.
    
    \begin{theorem}
    \label{thm m2}
		Let $(M^m, g, I,J,K)$ be a compact quaternionic K\"{a}hler manifold with smooth  boundary $\partial M$ and of quaternionic dimension $m\ge2$.
		Suppose that the scalar curvature is bounded from below by $16m(m+2)\kappa$ for some $\kappa \in \R$, and the second fundamental form on $\partial M$ is bounded from below by $\Lambda \in \R$.
		Let $\lambda_{1,p}(M)$ be the first Dirichlet eigenvalue of the $p$-Laplacian on $M$.
		Assume $1<p<\infty$, then
		\begin{align}\label{1.7}
		\lambda_{1,p}(M) \geq \bar{\lambda}_1(m,p,\k,\Lambda, R),
		\end{align}
		where $\bar{\lambda}_1(m,p,\k,\Lambda, R)$ is the first eigenvalue of the one-dimensional eigenvalue problem
        \begin{align}
        \label{1.8}
		(p-1)|\vp'|^{p-2}\vp''
		- \big(4(m-1)T_{\k,\Lambda}
		+3T_{4\k,\Lambda}\big)|\vp'|^{p-2} \vp' 
        = -\l|\varphi|^{p-2} \varphi	
        \end{align}
        on $[0, R]$ with boundary conditions $\varphi(0)=0$ and $\varphi'(R)=0$, and $R$ denotes the inradius of $M$.
    \end{theorem}
    The proof of Theorem \ref{thm m2} relies on a comparison theorem for the second derivatives of distance function to the boundary proved in \cite[Theorem 5.7]{LW23} and Barta's inequality. 

    The rest of the paper is organized as follows.
    In Section 2, we review some basic properties of quaternionic K\"ahler manifolds. 
    Section 3 and Section 4 are devoted to proving Theorem \ref{thm 1.3} and Theorem \ref{thm m2} respectively.
    In Section 5, we give an alternative proof of Theorem \ref{thm 1.1}.
    
    Throughout the paper, we denote the function $T_{\k}$ for $\k \in \R$ by
    \begin{align} \label{eqn Tk}
    T_{\k}(t) =\begin{cases}
        \sqrt{\k} \tan (\sqrt{\k}t), & \k >0,\\
		0, & \k =0,\\
		-\sqrt{-\k}	 \tanh(\sqrt{-\k}t), & \k <0,
    \end{cases}
    \end{align}
    and the function $c_{\k}$ is defined for $\k \in \R$ by
    \begin{align}
    \label{ck}
      c_{\k}(t) =
    \begin{cases}
        \cos (\sqrt{\k}t), & \k >0,\\
		1, & \k =0,\\
		\cosh(\sqrt{-\k}t), & \k <0.
    \end{cases}  
    \end{align}
    
\section{Quaternionic K\"ahler manifolds}

    In this section, we recall some basic properties about quaternionic K\"ahler manifolds. These properties are due to Berger \cite{Berger66} and Ishihara \cite{Ish74}. We shall follow the presentation in Kong, Li and Zhou \cite{KLZ08}.

    \begin{definition}
    A quaternionic K\"ahler manifold $(M^m, g)$ of quaternionic dimension $m$ (the real dimension is $4m$) is a Riemannian manifold with a rank three vector bundle $V \subset \text{End}(TM)$ satisfying the following properties.
    \begin{itemize}
        \item[(1)] In any coordinate neighborhood $U$ of $M$, there exists a local basis $\big \{I,J,K\big \}$ of $V$ such that
        \begin{align*}
            & I^2=J^2=K^2=-1,\\
            & IJ=-JI=K,\\
            & JK=-KJ=I,\\
            & KI=-IK=J,
        \end{align*}
        and for all $X,Y \in TM$
        \[
        \langle X,Y \rangle
        =\langle IX,IY \rangle
        =\langle JX,JY \rangle
        =\langle KX,KY \rangle.
        \]
        \item[(2)] If $\phi\in \Gamma(V)$, then $\nabla_X \phi \in \Gamma(V)$ for all $X\in TM$.
    \end{itemize}
    \end{definition}

    
    
     The Riemannian curvature tensor of $(M, g)$ is defined by
    \[
    R(X,Y,Z,W)
    =\langle 
    -\nabla_X\nabla_YZ
    +\nabla_Y\nabla_X Z
    +\nabla_{[X,Y]}Z
    ,W \rangle.
    \] In this paper, we are mostly concerned with the curvature properties of quaternionic K\"ahler manifolds.
    As in \cite{KLZ08} and \cite{BY22}, we define
    \begin{definition}
        Let $(M^m, g)$ be a quaternionic K\"ahler manifold.
        \begin{itemize}
            \item[(1)] The quaternionic sectional curvature of $M$ is defined as
            \[
            Q(X):=\frac{R(X,IX,X,IX)+R(X,JX,X,JX)+R(X,KX,X,KX)}{|X|^2}.
            \]
            \item[(2)] The orthogonal Ricci curvature of $M$ is defined as
            \[
            \operatorname{Ric}^{\perp}(X,X):=\operatorname{Ric}(X,X)-Q(X).
            \]
        \end{itemize}
    \end{definition}

    First of all, all quaternionic K\"ahler manifolds of quaternionic dimension $m \geq 2$ are Einstein (cf. \cite[Theorem 1.2]{KLZ08}), namely there exists a constant $\k$ such that 
    \[
    \text{Ric}=4(m+2)\k.
    \]
    Moreover, we have the following proposition.

   \begin{proposition}
   \label{prop 2.1}
    Let $(M^m, g)$ be a quaternionic K\"ahler manifold of quaternionic dimension $m \geq 2$ with $\mathrm{Ric}=4(m+2) \k$ for $\k \in \R$. Then
   $M$ has constant quaternionic sectional curvature 
        \[
        Q(X)=12 \kappa|X|^2,
        \]
    and constant orthogonal Ricci curvature
        \[
        \operatorname{Ric}^{\perp}(X, X)=4(m-1) \k|X|^2 .
        \]

   \end{proposition}
   \begin{proof}
      See Theorem 1.3 of \cite{KLZ08}.
   \end{proof}
    We end this section with the following useful lemma.
    \begin{lemma}
    Let $(M^m, g)$ be a quaternionic K\"ahler manifold of quaternionic dimension $m \ge 2$ with $\operatorname{Ric}=4(m+2) \k$ for $\k \in \R$. Let $\gamma:[a, b] \rightarrow M$ be a geodesic with unit speed and $X_I(t), X_J(t), X_K(t)$ are parallel vector fields along $\gamma$ such that $X_I(a)=I \gamma'(a), X_J(a)=J \gamma'(a), X_K(a)=K \gamma'(a)$. Then
    \[
    R(\gamma', X_I(t), \gamma', X_I(t))
    +R(\gamma', X_J(t), \gamma', X_J(t))
    +R(\gamma', X_K(t), \gamma', X_K(t))
    =12 \k
    \]
    for all $t \in[a, b]$.
    \end{lemma}
    \begin{proof}
        See Lemma 1.5 of \cite{KLZ08}.
    \end{proof}

\section{Modulus of Continuity Estimates}
In this section, we prove modulus of continuity estimates for solutions of a  fully nonlinear parabolic equation \eqref{thm 3.1} on a compact quaternionic K\"ahler manifold, which is the key step in the proof of Theorem \ref{thm 1.3}. 
    
    \begin{theorem}
    \label{thm 3.1}
    Let $(M^m, g, I,J,K)$ be a compact quaternionic K\"ahler manifold (possibly with non-empty  strictly convex boundary) of quaternionic dimension $m\ge2$ and diameter $D$, whose scalar curvature is bounded from below by $16m(m+2)\k$ for some $\k \in \mathbb{R}$.
    Suppose that $T\in(0,\infty]$, $1<p\leq 2$ and $v\in C^{2,1}(M \times [0, T))$ is a  solution of
    \begin{align}
	\label{3.1}
		\p_t v(x,t)
		=|\Delta_p v(x,t)|^{-\frac{p-2}{p-1}}\Delta_p v(x,t)
    \end{align}
    with Neumann boundary condition if $\p M\neq \emptyset$.
    Suppose $\varphi(s,t) \in C^{2,1}\big([0,D/2] \times [0,T)\big)$ satisfies the following properties
    \begin{itemize}
        \item[(1)] 
        $v(y,0)-v(x,0) \leq 2 \varphi(\frac{d(x,y)}{2},0), \  x,y \in M$;
		\item[(2)] 
        $0 \geq \varphi_t \geq 
        |\mathcal{F}\varphi|^{-\frac{p-2}{p-1}} \mathcal{F}\varphi,\	
        (s,t) \in [0,D/2]\times [0,T)$;
		\item[(3)] 
        $\varphi'(s,t)>0,\ (s,t) \in [0,D/2]\times [0,T)$;
		\item[(4)] $\varphi(0,t) \geq 0,\ t \in[0,T)$.
    \end{itemize}
    Then 
    \begin{align}
    \label{3.2}
		v(y,t)-v(x,t) \leq 2 \varphi(\frac{d(x,y)}{2},t)
    \end{align}
    for $(x,y, t)\in M\times M\times [0, T)$.
    Here one-dimensional operator $\mathcal{F}$ is defined by
    \begin{align*}
        \mathcal{F}\varphi
        :=(p-1)|\varphi'|^{p-2}\varphi''
        -( 4(m-1)T_{\k} +3T_{4\k} )|\varphi'|^{p-2}\varphi',	
    \end{align*}
    and  we denote  $\frac{\p^k \vp}{\p s^k}$ by $\vp^{(k)}$ for short throughout the paper.	
    \end{theorem}
	
    \begin{proof}
        The proof proceeds as in the K\"ahler case in \cite{WZ23} but slightly more involved.
		Let $\varepsilon >0$, and
		\begin{align*}
		C_{\varepsilon}(x,y,t)
		:=v(y,t)-v(x,t)-2\varphi(\frac{d(x,y)}{2},t)
		-\varepsilon e^t.
		\end{align*}
		To prove \eqref{3.2}, it suffices to show that $C_{\varepsilon}(x,y,t)<0$ for all $\varepsilon>0$.
    
        We argue by contradiction and assume that there exists  $(x_0,y_0,t_0)$ such that $C_{\varepsilon}$ attains its maximum zero on
		$M \times M \times [0,t_0]$ at $(x_0,y_0,t_0)$.
		Clearly $x_0 \neq y_0$, and $t_0 >0$. 
        If $\p M\neq \emptyset$, 
        similarly as in \cite[Theorem 3.1]{LW21},
		the strictly convexity of $\p M$, the Neumann condition and the positivity of $\varphi'$ rule out the possibility that $x_0 \in \partial M$ and $y_0 \in \partial M$.
		
		Compactness of $M$ implies that there exists an arc-length minimizing geodesic $\gamma_0$ connecting $x_0$ and $y_0$ such that $\gamma_0(-s_0)=x_0$ and $\gamma_0(s_0)=y_0$ with $s_0=d(x_0,y_0)/2$.
		We pick an orthonormal basis $ \{ e_i \}_{i=1}^{4m}$ for $T_{x_0}M$ with
        \[
        e_1 = \gamma'_0(-s_0),\ 
        e_2=I\gamma_0'(-s_0),\ 
        e_3=J\gamma_0'(-s_0),\ 
        e_4=K\gamma_0'(-s_0),
        \]
		and parallel transport it along $\gamma_0$ to produce an orthonormal basis $ \{ e_i(s) \}_{i=1}^{4m}$ for $T_{\gamma_0(s)} M$ with $e_1(s)=\gamma_0'(s)$ for each $s\in [-s_0, s_0]$.  
        
To calculate the derivatives,   we recall the first and second variation formulas of arc length.
        Suppose
		$\gamma(r, s) :
		[-\delta,\delta]\times [-s_0,s_0] \rightarrow M$ is a smooth variation of $\gamma_0$ such that $\gamma(0,s)=\gamma_0(s)$, then the variation formulas give
		\begin{align}
        \label{3.3}
		\frac{ \mathrm{d} }{ \mathrm{d} r} \big|_{r=0}
		L[\gamma(r,s)]
		=
		g( T,\gamma_r ) \big|_{-s_0}^{s_0},
		\end{align}
		and
		\begin{align}
        \label{3.4}
		\frac{ \mathrm{d}^2}{ \mathrm{d} r^2} \big|_{r=0}
		L[\gamma(r,s)]
		=
		\int_{-s_0}^{s_0}
		\big( | (\nabla_s \gamma_r)^{\perp} |^2-R(T,\gamma_r,T,\gamma_r) \big) \ \mathrm{d} s
		+g( T,\nabla_r \gamma_r )\big|_{-s_0}^{s_0},
		\end{align}
		where $T$ is the unit tangent vector to $\gamma_0$.
        
        First derivative inequality  yields
		\begin{align}
			\label{3.5}
			0 \leq \partial_t
			C_{\varepsilon}(x_0,y_0,t_0)
			=v_t (y_0,t_0)-v_t(x_0,t_0)
			-2\varphi_t(s_0, t_0) - \varepsilon e^{t_0},
		\end{align}
		and
		\begin{align*}
			\nabla v(x_0,t_0)=\varphi'(s_0,t_0)e_1(-s_0),      \quad
			\nabla v(y_0,t_0)=\varphi'(s_0,t_0)e_1(s_0).
		\end{align*}
		Then we get
		\begin{align}
        \label{3.6}			
            &\Delta_p v(y_0,t_0)-\Delta_p v(x_0,t_0)\\
		  =& (p-1) |\varphi'|^{p-2} 
        \big(v_{11}(y_0,t_0)-v_{11}(x_0,t_0) \big)
        +|\varphi'|^{p-2} 
        \sum_{i=2}^{4m}
	\big(
        v_{ii}(y_0,t_0)-v_{ii}(x_0,t_0) \big).\nonumber
		\end{align}
		To calculate the second derivative along $e_1$, we consider the variation $\gamma(r,s):=\gamma_0(s+r\frac{s}{s_0})$, then the formulas \eqref{3.3} and \eqref{3.4} give
        \[
        \frac{\mathrm{d}}{\mathrm{d} r} \big|_{r=0} L[\gamma(r)]=2,\quad \frac{\mathrm{d}^2}{\mathrm{d} r^2}
		\big|_{r=0} L[\gamma(r)] =0.
        \]
		Hence the second derivative test for this variation produces
		\begin{align}
		\label{3.9} 
        v_{11}(y_0,t_0)-v_{11}(x_0,t_0)-2 \varphi''(s_0,t_0) \leq 0.
		\end{align}
		
		To calculate the second derivative along $e_i$ ($i=2,3,4$), we consider the variation 
        \begin{align*}
        \gamma(r,s):=\operatorname{exp}_{\gamma_0(s)}
        ( r\eta(s)e_i(s) ),\ i=2,3,4,
        \end{align*}
        where $\eta(s)=\frac{c_{4\kappa}(s)}{c_{4\kappa}(s_0)}$ and $c_{\k}(s)$ is defined in (\ref{ck}). Clearly formulas \eqref{3.3} and \eqref{3.4} imply
        \[
        \frac{\mathrm{d}}{\mathrm{d} r} \bigg|_{r=0} L[\gamma(r)]=0,
        \]		
        and
		\[
		\frac{\mathrm{d}^2}{\mathrm{d} r^2} \bigg|_{r=0} L[\gamma(r)]
		= \int_{-s_0}^{s_0}
		\bigg( (\eta')^2 -\eta^2 R(e_1,e_i,e_1,e_i) \bigg) \
		\mathrm{d} s.
		\]
		So this variation produces
        \begin{align*}
        v_{ii}(y_0,t_0)-v_{ii}(x_0,t_0)
		-\varphi'(s_0,t_0) \int_{-s_0}^{s_0}
		\bigg( (\eta')^2 -\eta^2 R(e_1,e_i,e_1,e_i) \bigg) \ \mathrm{d} s \leq 0.
        \end{align*}
        Summing over $2\le i\le 4$, we get
		\begin{align*}
        &\sum_{i=2}^4
        \big(
        v_{ii}(y_0,t_0)-v_{ii}(x_0,t_0)
        \big) \\
		\leq 
        &
        \vp' 
        \int_{-s_0}^{s_0}
		\bigg( 3(\eta')^2 -\eta^2 \sum_{i=2}^4 R(e_1,e_i,e_1,e_i) \bigg)
		\ \mathrm{d} s\\
		=& 3\vp' \eta'\eta |_{-s_0}^{s_0}
		-\vp'\int_{-s_0}^{s_0}
		\eta^2 
        \big(
        \sum_{i=2}^4
        R(e_1,e_i,e_1,e_i)-12\kappa
        \big)
		\ \mathrm{d} s\\
   	\le&  -6 T_{4\kappa}(s_0)\vp',
		\end{align*}
        where we have used the assumption that $\vp'(s,t)>0$ and $\sum_{i=2}^4 R(e_1,e_i,e_1,e_i)\ge 12 \kappa$, in the view of Proposition \ref{prop 2.1}.
        Therefore we have
		\begin{equation}
		  \label{3.11}
		  \sum_{i=2}^4 
            \big(
            v_{ii}(y_0,t_0)-v_{ii}(x_0,t_0)
            \big)
            \leq
		  -6 T_{4\kappa}(s_0) \varphi'(s_0,t_0).
		\end{equation}
        		
        To calculate the second derivative along $e_i
        $ ($5\le i\le 4m$), we  consider the variation
        \[
        \gamma(r,s): = \operatorname{exp}_{\gamma_0(s)}( r\zeta(s)e_i(s) ),
        \]		
        where $\zeta(s)=\frac{c_{\kappa}(s)}{c_{\kappa}(s_0)}$.
		Similarly, the second variation formula gives
		\[
		v_{ii}(y_0,t_0)-v_{ii}(x_0,t_0)
		\leq
        \varphi'(s_0,t_0)
		\int_{-s_0}^{s_0}
		\bigg( (\zeta')^2 -\zeta^2 R(e_1,e_i,e_1,e_i) \bigg) \ \mathrm{d} s.
		\]
		Summing over $5\le i\le 4m$, we obtain
        \begin{align*}
        &\sum_{i=5}^{4m}
		\big(v_{ii}(y_0,t_0)-v_{ii}(x_0,t_0)\big)\\
        \leq & 
        -8(m-1) T_{\kappa}(s_0) \vp'
	  -\vp'\int_{-s_0}^{s_0}
	  \zeta^2
	  \big(
	  \sum_{i=5}^{4m}
        R(e_1,e_i,e_1,e_i)-4(m-1)\kappa
	  \big)
	  \mathrm{d} s\\
        \le &  -8(m-1) T_{\kappa}(s_0)\vp',
    \end{align*}
    where in the last inequality we have used 
    \[
    \sum_{i=5}^{4m}R(e_1,e_i,e_1,e_i) = \operatorname{Ric}^{\perp} (e_1,e_1) \ge 4(m-1)\kappa,
    \]
    in the view of Proposition \ref{prop 2.1}.
		Thus we get
		\begin{equation}
			\label{3.12}
			\sum_{i=5}^{4m}
			\big(v_{ii}(y_0,t_0)-v_{ii}(x_0,t_0)\big)
			\leq
			-8(m-1) T_{\kappa}(s_0) \varphi'(s_0,t_0).
		\end{equation}
		Plugging inequalities \eqref{3.9}, \eqref{3.11} and \eqref{3.12} into equation \eqref{3.6}, we have
		\begin{equation}
			\label{3.13}
			\Delta_p v(y_0,t_0)-\Delta_p v(x_0,t_0)
			\leq
			2\mathcal{F} \varphi(s_0,t_0).
		\end{equation}
		Let $h(t):=|t|^{-\frac{p-2}{p-1}}t$, which is odd, increasing and convex for all $t>0$. We rewrite inequality  \eqref{3.13} as
		\[
		h\big(\Delta_p v(y_0,t_0)\big)
		\leq
		h\big(\Delta_p v(x_0,t_0) +2\mathcal{F} \varphi\big).
		\]
		Applying Lemma 2.1 of  \cite{LW21-c} with
		$t=\Delta_p v(x_0,t_0)$ and
		$\delta= -\mathcal{F} \varphi\geq0$, we get
		\[
		h\big(
        \Delta_p v(x_0,t_0) +2\mathcal{F} \varphi
        \big)
		\leq
		h
        \big(\Delta_p v(x_0,t_0))+2h(\mathcal{F} \varphi
        \big),
		\]
		that is
		\begin{equation}
			\label{3.14}
			\frac{1}{2}\Big(
			|\Delta_p v|^{-\frac{p-2}{p-1}}\Delta_p v    \big|_{(y_0,t_0)}
			-|\Delta_p v|^{-\frac{p-2}{p-1}}\Delta_p v \big|_{(x_0,t_0)}
			\Big)
			\leq  |\mathcal{F}\varphi|^{-\frac{p-2}{p-1}}
			\mathcal{F}\varphi.
		\end{equation}
		Combining \eqref{3.5} and \eqref{3.14}, we get
		\[
		\varphi_t \leq
		|\mathcal{F}\varphi|^{-\frac{p-2}{p-1}}\mathcal{F}\varphi
		-\frac{\varepsilon}{2} e^{t_0}
		< 
        |\mathcal{F}\varphi|^{-\frac{p-2}{p-1}}\mathcal{F}\varphi,
		\]
		which contradicts the inequality in assumption (2). Therefore
        $$C_\varepsilon(x,y,t)<0$$
        holds true for $(x,y, t)\in M\times M\times [0,T)$. Hence we complete the proof of Theorem \ref{thm 3.1}.
	\end{proof}	

        As in the K\"ahler case, the modulus of continuity estimates imply lower bounds for the first nonzero eigenvalue of the $p$-Laplacian on a quaternionic K\"ahler manifold. We restate Theorem \ref{thm 1.3} here. 

        \begin{theorem}
        Let $(M^{m},g,I,J,K)$ be a compact quaternionic K\"{a}hler manifold of quaternionic dimension $m\ge2$ and diameter $D$, whose scalar curvature is bounded from below by $16m(m+2)\k$ for some $\k \in \R$.
        Let $\mu_{1,p}(M)$ be the first nonzero eigenvalue of the $p$-Laplacian on M (with Neumann boundary condition if $M$ has a strictly convex boundary). Assume $1<p\leq 2$, then
	\begin{align*}
		\mu_{1,p}(M) \geq \bar{\mu}_{1}(m,p,\kappa,D),
	\end{align*}
        where $\bar{\mu}_{1}(m,p,\kappa,D)$ is the first nonzero Neumann eigenvalue of the one-dimensional eigenvalue problem
	\begin{align}
	\label{3.15}
		(p-1)|\varphi'|^{p-2}\varphi''
		-( 4(m-1)T_{\k} +3T_{4\k} )|\varphi'|^{p-2}\varphi'
		= -\mu |\varphi|^{p-2} \varphi
	\end{align}
	on interval $[-D/2,D/2]$.
	\end{theorem}

        \begin{proof}
            The proof is a slight modification of the proof in the K\"ahler case in \cite{WZ23}, so we leave the details to interested readers.
        \end{proof}

\section{The First Dirichlet Eigenvalue}
    
    In this section, we give the proof of Theorem \ref{thm m2}.
    The key tools are a comparison theorem for the distance function to the boundary, and  Barta's inequality.
    Let $M$ be a compact manifold with smooth boundary, and define the distance function to the boundary of $M$ by
    \begin{align*}
    d(x, \partial M)=\inf \Big \{ d(x,y): y \in \partial M \Big\},
    \end{align*}
    and the inradius of $M$ by $R=\mathop{\sup}\limits_{x\in M} d(x, \p M)$. 
 
    By choosing $\alpha=(p-1)|\nabla \varphi|^{p-2}$ and $\beta=|\nabla \varphi|^{p-2}$ in Theorem 5.7 of \cite{LW23}, we have the following lemma.
    \begin{lemma}
	\label{lm Dirichlet eigenvalue}
		Let $(M^m, g, I,J,K)$ be a compact quaternionic K\"ahler manifold with smooth  boundary, and $R$ be the inradius of $M$.
	    Suppose that the scalar curvature is bounded from below by $16m(m+2)\k$ for some $\kappa \in \R$, and the second fundamental form on $\partial M$ is bounded from below by $\Lambda \in \R$.
	  Let $p\in(1, \infty)$, and assume $\varphi$ is a smooth function on $[0,R]$ satisfying $\varphi'\geq 0$.
	    Then for any smooth function $\psi$ satisfying
        \begin{align*}
		\psi(x) \leq \varphi(d(x,\partial M)) 
        \ \ \text{for} \ \ x \in M,\ \ \text{and} \ \ \psi(x_0) =\varphi(d(x_0,\partial M)),
		\end{align*}
	it holds
    \begin{align*}
	    \Delta_p \psi (x_0)
	    \leq
        \bar{ \mathcal{F} }\varphi(d(x_0,\partial M)).
    \end{align*}
    Here one-dimensional operator $\bar{\mathcal{F}}$ is defined by
    \begin{align*}
	\bar{\mathcal{F}}\varphi
	=(p-1)|\varphi'|^{p-2}\varphi''
	-\big(4(m-1)T_{\k,\Lambda}
	+3T_{4\k,\Lambda}\big) |\varphi'|^{p-2} \varphi'
    \end{align*}
    for all $\varphi\in C^2([0, R])$, and $T_{\kappa,\Lambda}$ is defined by \eqref{1.6}.
    \end{lemma}
	
    Consider the following  one-dimensional eigenvalue problem on $[0, R]$ that
    \begin{equation}
    \label{4.1}
    \left \{
    \begin{aligned}
			& \bar{\mathcal{F}}\varphi
			= -\lambda |\varphi|^{p-2} \varphi,\\
			&\varphi(0)=0,\ \varphi'(R)=0,
    \end{aligned}
    \right.
    \end{equation}
    and denote by $\bar{\lambda}_1(m,p,\k,\Lambda, R)$ (written as $\bar{\lambda}_1$ for short)  the first eigenvalue of  problem \eqref{4.1}. Then it is easy to see that $\bar{\lambda}_1$ can be characterized by
    \begin{align*}
	\bar{\lambda}_1
	=\inf
	\bigg\{
	\frac
	{ \int_0^{R} |\phi'|^p C_{\k,\Lambda}^{4m-4} C_{4\k,\Lambda}^3 \ \mathrm{d} s }
	{ \int_0^{R} |\phi|^p C_{\k,\Lambda}^{4m-4}    C_{4\k,\Lambda}^3 \ \mathrm{d} s }
	\ \big|\phi \in W^{1,p}\big((0,R)\big)\backslash\{0\},\  \phi(0)=0
	\bigg\},
    \end{align*}
    where $C_{\kappa,\Lambda}$ is defined by \eqref{1.5}.
	
	\begin{proof}[Proof of Theorem \ref{thm m2}]
	    Let $\varphi$ be an eigenfunction corresponding to $\bar\lambda_1$, then
	    \[
		\bar{\mathcal{F}}\varphi
		= -\bar{\lambda}_1 |\varphi|^{p-2} \varphi
		\]
	    with $\varphi(0)=0$ and $\varphi'(R)=0$.
	    Similarly as in the proof of \cite[Lemma 4.2]{LW21}, we can choose $\varphi$ such that $\varphi(s) >0$ on $(0,R]$, and $\varphi'(s)>0$ on $[0, R)$. Define a trial function for $\lambda_{1,p}(M)$ by
	    \[
		v(x):=\varphi(d(x,\partial M)),
		\]
	  then Lemma \ref{lm Dirichlet eigenvalue} gives
		\begin{align}\label{4.2}
		\Delta_p v(x)
		\leq
		\bar{\mathcal{F}}\varphi(d(x,\partial M))
		\end{align}
away from the cut locus of $\p M$, and thus globally in the distributional sense, see \cite[Lemma 5.2]{LW21-b}.
	    Recall that $\varphi$ is an eigenfunction with respect to $\bar\lambda_1$, namely
	    \begin{align}\label{4.3}
		\bar{\mathcal{F}}\varphi(d(x,\partial M))
		=
		-\bar{\lambda}_1 |v(x)|^{p-2} v(x),
		\end{align}
		thus we conclude from \eqref{4.2} and \eqref{4.3} that
	    \begin{align*}
		\Delta_p v
		\leq
		-\bar{\lambda}_1 |v|^{p-2} v, \quad x\in M.
	    \end{align*}
        By the construction  of $v(x)$, we see $v(x)>0$ for $x\in M$, and $v(x)=0$ for $x\in \p M$. Thus we conclude from Barta's inequality (cf. \cite[Theorem 3.1]{LW21-b}) that
		\[
		\lambda_{1,p}(M) \geq \bar{\lambda}_1.
		\]
		The proof of  Theorem \ref{thm m2} is complete.
	\end{proof}

\section{Elliptic Proof of Theorem \ref{thm 1.1}}

    By using modulus of continuity for solutions to the heat equation, Li and the first author of this paper \cite{LW23} proved the lower bound estimates for the first nonzero eigenvalue on quaternionic K\"ahler manifolds (see Theorem \ref{thm 1.1}).
    In this section,  following Ni's method (see Section 6 of \cite{Ni13}), we give an elliptic proof of Theorem \ref{thm 1.1}.

    \begin{proof}[Elliptic proof of Theorem \ref{thm 1.1}]
   Let $u$ be the eigenfunction with respect to the first nonzero closed eigenvalue (or Neumann eigenvalue), and $\vp$ be the first Neumann eigenfunction of (\ref{thm 1.2}), and we consider the quotient of the oscillations of $u$ and $\vp$ 
    \[
    Q(x,y):=\frac{u(y)-u(x)}{\vp \big( \frac{d(x,y)}{2} \big) }
    \]
    on $\overline{M}\times\overline{M} \setminus \Delta$,
   where $\Delta:=\{ (x,x):x \in \overline{M} \}$ denotes the diagonal. 
    The function $Q$ can be extended to a set $(\overline{M}\times\overline{M} \setminus \Delta) \cup UM$. Here $UM:=\{(x,X):x\in M, |X|=1\}$. On $UM$, with the extension $Q(x,X)$ defined as
    \[
    Q(x,X):=\frac{ 2g(\nabla u(x),X) }{\vp'(0)}.
    \]
    Now we divide the proof into two cases.
    
    \textbf{Case 1.} the maximum of $Q$, which is clearly positive and denoted by $G$, is attained at some $(x_0, y_0)$ with $x_0 \neq y_0$. The Neumann condition and strict convexity of $M$ forces that both $x_0$ and $y_0$ must be in $M$. Indeed if $x_0\in \partial M$, then taking derivative along the exterior normal direction $\nu$ at $x_0$ yields
    \begin{align*}
    \frac{\partial}{\partial \nu}\bigg|_{x_0} Q(x,y_0)
    &=
    \frac{\partial}{\partial \nu}\bigg|_{x_0}
    \frac{u(y_0)-u(x)}{\vp \big( \frac{d(x,y_0)}{2} \big) }
    \\
    &=
    -\frac{G \vp'(\frac{d(x_0,y_0)}{2})  }
    { 2 \vp(\frac{d(x_0,y_0)}{2}) } 
    \frac{\partial}{\partial \nu} \big|_{x_0}
    d(x,y_0)\\
    &<0,    
    \end{align*}
    which is a contradiction with the maximum assumption. Here in the last inequality, we used the convexity assumption of $M$.

    As in the proof of Theorem \ref{thm 3.1}, there exists a length-minimizing geodesic $\gamma_0:[-s_0,s_0]\rightarrow M$ such that it connects $x_0$ and $y_0$ with $s_0=d(x_0,y_0)/2$. We can also choose an orthonormal basis $\{e_i(s)\}_{i=1}^{4m}$ for $T_{\gamma_0(s)}M$ with $e_1(s)=\gamma_0'(s)$ for each $s\in [-s_0, s_0]$ and 
    \[
    e_2(-s_0)=I\gamma_0'(-s_0),\quad 
    e_3(-s_0)=J\gamma_0'(-s_0),\quad 
    e_4(-s_0)=K\gamma_0'(-s_0).
    \]
    Then first derivative gives
    \[
    \nabla u(x_0)=\frac{G}{2} \vp'(s_0) e_1(-s_0), \quad
    \nabla u(y_0)=\frac{G}{2} \vp'(s_0) e_1(s_0). 
    \]

    For second derivatives, from
    \[
    \frac{\mathrm{d}^2}{ \mathrm{d} r^2 }\bigg|_{r=0}
    Q(\gamma_0(-s_0-r),\gamma_0(s_0+r)) \leq 0,
    \]
    we have
    \begin{align*}
    0
    \geq& 
    \frac{ u_{11}(y_0)-u_{11}(x_0) }{ \vp(s_0) }
    -\frac{ u(y_0)-u(x_0) }{ \vp^2(s_0) } \vp''(s_0)\\
    &- 2 \frac{ u_1(y_0)+u_1(x_0) }{ \vp^2(s_0) } \vp'(s_0)
    +2G \big( \frac{ \vp'(s_0) }{ \vp(s_0) } \big)^2,
    \end{align*}
    thus we have
    \begin{equation}
    \label{eqn e1}
    0 \geq u_{11}(y_0)-u_{11}(x_0)
    -G\vp''(s_0) .
    \end{equation}
    For $i=2,3,4$, we consider the variation
    \[
    \gamma(r,s)=\text{exp}_{\gamma_0(s)}(r\eta(s)e_i(s)), 
    \]
    where $\eta(s)=\frac{c_{4\k}(s)}{ c_{4\k}(s_0) }$. This variation gives
    \[
    u_{ii}(y_0)-u_{ii}(x_0)
    -\frac{G}{2} \vp'(s_0) 
    \int_{-s_0}^{s_0} (\eta')^2-\eta^2 R(e_1,e_i,e_1,e_i)\ \mathrm{d}s \leq 0.
    \]
    For $5\leq i \leq 4m$, we consider the variation 
    \[
    \gamma(r,s)=\text{exp}_{\gamma_0(s)}(r\zeta(s)e_i(s)), 
    \]
    where $\zeta(s)=\frac{ c_{\k}(s) }{ c_{\k}(s_0) }$. 
    This variation produces 
    \[
    u_{ii}(y_0)-u_{ii}(x_0)
    -\frac{G}{2} \vp'(s_0) 
    \int_{-s_0}^{s_0} (\zeta')^2-\zeta^2 R(e_1,e_i,e_1,e_i) \ \mathrm{d}s \leq 0.
    \]

    Summing over $2\leq i \leq 4$, we obtain
    \[
    \sum_{i=2}^4 \big( u_{ii}(y_0)-u_{ii}(x_0) \big)
    -\frac{G}{2} \vp'(s_0) 
    \int_{-s_0}^{s_0} 3(\eta')^2
    -\eta^2 \sum_{i=2}^4 R(e_1,e_i,e_1,e_i)\ \mathrm{d}s \leq 0.
    \]
    Using the assumption $\sum_{i=2}^4 R(e_1,e_i,e_1,e_i)\geq 12 \k$, we have
    \begin{equation}
    \label{eqn e2}
    \sum_{i=2}^4 \big( u_{ii}(y_0)-u_{ii}(x_0) \big)
    +3G \vp'(s_0) T_{4\k}(s_0) \leq 0.
    \end{equation}
    Similarly, summing over $5\leq i \leq 4m$, we get
    \begin{equation}
    \label{eqn e3}
    \sum_{i=5}^{4m} \big( u_{ii}(y_0)-u_{ii}(x_0) \big)
    +4(m-1)G \vp'(s_0) T_{\k}(s_0)\leq 0.
    \end{equation}

    Combining (\ref{eqn e1}), (\ref{eqn e2}) and (\ref{eqn e3}), we have
    \[
    \Delta u(y_0)-\Delta u(x_0) -G \vp''(s_0)+3GT_{4\k}(s_0)\vp'(s_0)+4G(m-1)T_{\k}(s_0)\vp'(s_0)\leq 0,
    \]
    that is 
    \[
    -\mu_{1,2}(M) (u(y_0)-u(x_0)) +G \bar{\mu}_1(m,2,\k,D) \vp(s_0) \leq 0,
    \]
    which proves $\mu_{1,2}(M) \geq \bar{\mu}_{1}(m,2,\k,D)$.

    \textbf{Case 2.} the maximum of $Q$ is attained at some $(x_0,X_0)\in UM$. 
    It is easy to see that $X_0= \frac{\nabla u(x_0) }{|\nabla u(x_0)|}$ and $G=2|\nabla u(x_0)|$. By the assumption, we know $x_0\in M$. In fact if $x_0\in \partial M$, then taking derivarive along exterior normal direction $\nu$ at $x_0$ yields
    \begin{align*}
       \frac{\partial}{\partial \nu}\bigg|_{x_0} |\nabla u (x)|^2
    &=2 g(\nabla_{\nabla u(x)} \nabla u(x),\nu )\big|_{x_0}\\
    &=-2 \operatorname{II} (\nabla u(x_0), \nabla u(x_0))<0,    
    \end{align*}
    contradicting with the maximum assumption, where $\operatorname{II}$ is the second fundamental form of $M$ at $x_0$.

    Now pick up an orthonormal frame $\{e_i\}_{i=1}^{4m}$ at $x_0$ so that $e_1=X_0$.  
    We also parallel transport it to a neighborhood of $x_0$.

    Since $|\nabla u(x)|^2$ attains its maximum at the interior point $x_0$, we have
    \[
    \nabla |\nabla u|^2(x_0)=0,
    \]
    hence 
    \[
    u_{1k}(x_0)=0
    \]
    for any $1\leq k \leq 4m$. Moreover, the maximum principle concludes for any $2\leq k \leq 4m$, 
    \begin{equation}
    \label{eqn e4}
        0\geq u_1 u_{kk1}+|u_1|^2 R(e_1,e_k,e_1,e_k).
    \end{equation}

    Let $x(s)=\text{exp}_{x_0}(-se_1),y(s)=\text{exp}_{x_0}(se_1)$ and $h(s)=Q(x(s),y(s))$. Since $Q$ achieves its  maximum at $(x_0,X_0)$, we have that 
    \[
    h(s)\leq h(0)=G, \quad s\in(-\varepsilon,\varepsilon),
    \]
    which implies that $\lim\limits_{s\rightarrow 0} h'(s)=0$ and $\lim\limits_{s\rightarrow 0} h''(s)\le0$.

    Direct calculation shows that
    \[
    h'(s)
    =\frac{ g(\nabla u(y(s)), y'(s)) 
    - g(\nabla u(x(s)), x'(s))}
    {\vp(s)}-h(s) \frac{\vp'(s)}{\vp(s)}
    \]
    and
    \begin{align*}
       h''(s)
       =&\frac{\nabla_{e_1 e_1}^2 u(y(s))-\nabla_{e_1 e_1}^2 u(x(s))}{\vp(s)}\\
       &-2 \frac{g(\nabla u(y(s)), y'(s)) 
       -g(\nabla u(x(s)), x'(s))}{\vp(s)} \frac{\vp'(s)}{\vp(s)}\\
       &-h(s) \frac{\vp''(s)}{\vp(s)}+2 h(s)(\frac{\vp'(s)}{\vp(s)})^2\\
       =&\frac{\nabla_{e_1 e_1}^2 u(y(s))-\nabla_{e_1 e_1}^2 u(x(s))}{\vp(s)}
       -2\frac{h'(s)}{\vp(s)}\vp'(s)
       -h(s)\frac{\vp''(s)}{\vp(s)}.
    \end{align*}
    Observing that $\lim\limits_{s \rightarrow 0} \frac{h'(s)}{\vp(s)}=h''(0)$, 
    the second equation implies that
    \[
    h''(0)=2 u_{111}(x_0)-2 h''(0)-G \lim\limits_{s \rightarrow 0} \frac{\vp''(s)}{\vp(s)}.
    \]
    From Equation (\ref{1.3}), it follows
    \[
    \lim\limits_{s \rightarrow 0} \frac{\vp''(s)}{\vp(s)}
    =\lim\limits_{s \rightarrow 0}
    \big( 4(m-1)T_{\k}(s)+3T_{4\k}(s) \big) 
    \frac{\vp'(s)}{\vp(s)}-\bar{\mu}_1
    =4(m-1)\k+12\k-\bar{\mu}_1.
    \]
    Then we have
    \begin{equation}
    \label{eqn e5}
    2u_{111}(x_0) -G( 4(m-1)\k+12\k ) +G \bar{\mu}_1 \leq 0.
    \end{equation}
    Combining (\ref{eqn e4}) and (\ref{eqn e5}), we derive
    \[
    2g( \nabla u, \nabla \Delta u) +2 |u_1|^2 \sum_{k=2}^{4m} R(e_1,e_i,e_1,e_i) -2 |u_1|^2 (4(m-1)\k+12\k) +2 |u_1|^2 \bar{\mu}_1 \leq 0,
    \]
    which also implies $\mu_{1,2}(M) \geq \bar{\mu}_{1}(m,2,\k,D)$.
    \end{proof}

    \begin{remark}
   Similarly as in \cite{SWW19, ZW17}, the elliptic proof also works for Theorem \ref{thm 1.3} when $1<p<2$.
    \end{remark}
    
	\bibliographystyle{plain}
	\bibliography{ref}
	
\end{document}